\documentclass[12pt]{article}
\usepackage[utf8]{inputenc}
\usepackage[a4paper, total={6in, 8in}]{geometry}
\usepackage{graphicx}
\usepackage{latexsym}
\usepackage{amsmath}
\usepackage{amssymb}
\usepackage{amsthm}
\usepackage{bibentry}
\usepackage{hyperref}
\usepackage{todonotes}
\usepackage{xcolor}

 \newtheorem{thm}{Theorem}[section]
\newtheorem{theorem}[thm]{Theorem}

\newtheorem{proposition}[thm]{Proposition}

\newtheorem{rmk}[thm]{Remark}

\newtheorem{definition}[thm]{Definition}
\newtheorem{corollary}[thm]{Corollary}

\newtheorem{problem}[thm]{Problem}
\newcommand{\ZZ}{\mathbb{Z}}

\newcommand{\CC}{\mathbb{C}}
\newcommand{\RR}{\mathbb{R}}
\newcommand{\PP}{\mathbb{P}}
\title{Cohomologous symplectic forms with different Gromov widths}
\author{Shengzhen Ning}
\date{}
\AtEndDocument{\bigskip{\footnotesize
		\textsc{School of Mathematics, University of Minnesota, Minneapolis, MN, US} \par
		\textit{E-mail address}: \texttt{ning0040@umn.edu} \par
		
}}

\begin{document}

\maketitle
\begin{abstract}
    We study McDuff-Salamon's Problem 46 in \cite{MSbook} by showing that there exist closed manifolds of dimension $\geq 6$ admitting cohomologous symplectic forms  with different Gromov widths. The examples are motivated by Ruan's early example (\cite{Ruan3fold}) of deformation inequivalent symplectic forms in dimension $6$ distinguished by Gromov-Witten invariants. To find cohomologous symplectic forms and compare their Gromov width, we make use of Li-Liu's theorem of symplectic cone for manifolds with $b_2^+=1$ and Biran's ball packing theorem in dimension $4$. Along the way, we also show that these cohomologous symplectic forms can have distinct first Chern classes, which answers another question by Salamon in \cite{Salamonsurvey}.
\end{abstract}
\section{Introduction}
 The Gromov width of a $2n$-dimensional symplectic manifold $(M,\omega)$ is defined by
 \[w_G(M,\omega)=\sup\{\pi r^2\,|\,B^{2n}(r)\text{ can be symplectically embedded into }(M,\omega)\},\]
 where $B^{2n}(r)$ denotes the open ball in $\RR^{2n}$ of radius $r$ equipped with the standard symplectic form $\sum_{i=1}^ndx_i\wedge dy_i$.
 Problem 46 in the open problem chapter of \cite{MSbook} asks the following.
\begin{problem}
 Is there a closed manifold $M$ with cohomologous symplectic forms $\omega_0,\omega_1$ such that $(M,\omega_0)$ and $(M,\omega_1)$ have different Gromov widths?
\end{problem}



Another problem related to cohomologous symplectic forms on a fixed closed manifold, raised in \cite{Salamonsurvey}, concerns their first Chern classes. In dimension $4$, it is a consequence of Seiberg-Witten theory that cohomologous symplectic forms must have the same first Chern class (\cite[Corollary A]{Salamonsurvey}). In higher dimensions, one can ask the following question.

\begin{problem}[\cite{Salamonsurvey} Discussion 4.6]\label{prob:Salamon}
Do cohomologous symplectic forms on a closed manifold of dimension $\geq 6$ always have the same first Chern class?
\end{problem}

The aim of this note is to provide examples to both of the above problems in dimension $\geq 6$.

\begin{theorem}\label{thm:main}
    Let $X$ be any symplectic $4$-manifold homeomorphic but not diffeomorphic to the rational manifold $\CC\PP^2\#k\overline{\CC\PP^2}$ with $k\geq 1$. Then there exists some symplectic form $\omega_X$ on $X$ and $\omega_{S^2}$ on $S^2$ such that the $6$-manifold $M:=X\times S^2$ admits symplectic form cohomologous to the product symplectic form $\omega_X\oplus \omega_{S^2}$ but with smaller Gromov width. Moreover, the first Chern classes of these cohomologous symplectic forms are different.
\end{theorem}

    Symplectic $4$-manifolds $X$ satisfying the conditions in Theorem \ref{thm:main} are quite abundant in the realm of exotic $4$-manifolds, even under the minimality assumption. For example, see \cite{Kotschick} for $k=8$, \cite{Park7} for $k=7$, \cite{SS6} for $k=6$, \cite{PSS5} for $k=5$, \cite{RU4} for $k=4$, \cite{AP3,BK3} for $k=3$ and \cite{AP2pt} for $k=2$. Note that $k=2$ is the current world record for smallest $k$. We thus have the following immediate corollary.

\begin{corollary}
    For any $k\geq 2$, the smooth $6$-manifold $M:=(\CC\PP^2\#k\overline{\CC\PP^2})\times S^2$ admits cohomologous symplectic forms with different Gromov widths and different first Chern classes.
\end{corollary}

    As we will see in the proof, based on the examples in Theorem \ref{thm:main}, it is easy to construct cohomologous symplectic forms with different Gromov widths and Chern classes in any dimension $\geq 6$.


For Problem \ref{prob:Salamon}, the examples in Theorem \ref{thm:main} are not the only candidates. Indeed, $K3\times S^2$ also works. One can use various symplectic homotopy K3 surfaces such as the ones constructed in \cite{Gompf,FS,VidussiK3}, which have non-vanishing first Chern classes. The knot surgery manifolds $E(n)_K$ of the elliptic surface $E(n)$ have been investigated in \cite{IP99}. They showed that, for two fibered knots $K,K'$ with the same genus, the diffeomorphic $6$-manifolds $E(n)_K\times S^2$ and $E(n)_{K'}\times S^2$ can have deformation inequivalent cohomologous symplectic forms with the same first Chern class. This was later generalized in \cite{Herrera} into the $S^2$-bundle settings. Rather than comparing symplectic forms coming from different knots $K,K'$ as \cite{IP99}, one can actually directly compare the symplectic form constructed from $E(2)_K\times S^2$ directly with the product symplectic form on $K3\times S^2$. See the discussion after the proof of Theorem \ref{thm:main} in the last section.

\section*{Acknowledgement}
The author would like to thank his advisor Tian-Jun Li for helpful discussions and comments during the preparation of this manuscript.

\section{Dimension $4$}

\subsection{Exceptional classes and minimal models}
For a symplectic $4$-manifold $(X,\omega)$, let $\mathcal{E}_X$ (resp. $\mathcal{E}_{\omega}$) be the set of classes in $H_2(X;\ZZ)$ which can be represented by smoothly (resp. symplectically) embedded sphere of self-intersection $-1$. A deep result coming from Taubes-Seiberg-Witten theory is that $\mathcal{E}_X$ is empty if and only if $\mathcal{E}_{\omega}$ is empty (\cite{Taubes},\cite{LiLiuruled},\cite{Li99}). In this case, $(X,\omega)$ is said to be {\bf minimal}. 

Assume $(X,\omega)$ is minimal and not diffeomorphic to a rational or ruled manifold. By the symplectic blowup construction, the smooth manifold $\tilde{X}:=X\#l\overline{\CC\PP^2}$ also admits symplectic forms. Let $\tilde{\omega}$ be any symplectic form on $\tilde{X}$. In general, it is unknown whether $\tilde{\omega}$ is symplectomorphic to a symplectic form coming from the symplectic blowup construction. However, by \cite[Corollary 3]{Li99}, we always have 
\begin{equation}\label{equation:-1}
    \mathcal{E}_{\tilde{\omega}}=\{E_1,\cdots,E_l\},\,\,\mathcal{E}_{\tilde{X}}=\{\pm E_1,\cdots,\pm E_l\}
\end{equation}
where $E_i$'s are the line classes in those $l\overline{\CC\PP^2}$. In particular, $E_i\cdot E_j=0$ for any $i\neq j$. By standard pseudoholomorphic curve technique, the classes in $\mathcal{E}_{\tilde{\omega}}$ can be represented by disjoint embedded $J$-holomorphic exceptional spheres for generic $\tilde{\omega}$-tame almost complex structure $J$. Moreover, the space of the symplectic representatives of any class in $\mathcal{E}_{\tilde{\omega}}$ is connected. By performing symplectic blowdown operations on those exceptional spheres, one can obtain the {\bf minimal model} for $(\tilde{X},\tilde{\omega})$. Indeed, (\ref{equation:-1}) is closely related to the uniqueness of minimal models for symplectic $4$-manifolds not diffeomorphic to rational or ruled manifolds.

\begin{theorem}[\cite{McDuffrationalruled,Mcduffimmersed}]\label{thm:minimalmodel}
Any symplectic $4$-manifold has a minimal model. When $X$ is not diffeomorphic to rational or ruled manifolds, the symplectomorphism type of the minimal model is unique.
\end{theorem}

Note that the minimal model depends on the choice of the symplectic form on the non-minimal manifold. As pointed out in \cite[Example 13.4.5]{MSbook}, unless the manifold is rational or ruled, different choice of symplectic forms might even potentially result in non-diffeomorphic manifolds after blowdown. This is because in general, it is unknown whether the diffeomorphism group acts transitively on the space of smoothly embedded exceptional spheres.



\subsection{Li-Liu's result of symplectic cones of $4$-manifolds with $b_2^+=1$}
Now, we review the result from \cite{LiLiu01}. For an oriented closed smooth $4$-manifold $X$ admitting symplectic structures, introduce the set consisting all the possible symplectic canonical classes  
\[\mathcal{K}_X:=\{K_{\omega}\,|\,\omega\text{ is a symplectic form on }X\}\subseteq H^2(X;\ZZ).\]
For any $K\in \mathcal{K}_X$, consider the {\bf $K$-symplectic cone}
\[\mathcal{C}_K:=\{[\omega]\,|\,\omega\text{ is a symplectic form on }X\text{ and }K_{\omega}=K\}\subseteq H^2(X;\RR),\]
and the set of {\bf $K$-exceptional classes}
\[\mathcal{E}_K:=\{E\in \mathcal{E}_X\,|\,E\cdot K=-1\}.\]
When $b_2^+(X)=1$, the {\bf positive cone} \[\mathcal{P}:=\{a\in H^2(X;\RR)\,|\,a^2>0\}\subseteq H^2(X;\RR)\]
will contain two connected components. When some $K\in\mathcal{K}_X$ is chosen, the connected component of $\mathcal{P}$ containing $\mathcal{C}_K$ will be called the {\bf forward cone associated to $K$}. We denote it by $\mathcal{FP}(K)$.
\begin{theorem}[\cite{LiLiu01}]\label{thm:liliucone}
    If $b_2^+(X)=1$ and $\mathcal{K}_X\neq \emptyset$, then
    \[\mathcal{C}_K=\{a\in \mathcal{FP}(K)\,|\,a\cdot E> 0\text{ for any }E\in\mathcal{E}_K\}.\]
\end{theorem}

In particular, when $X$ is minimal and $a\in H^2(X;\RR)$ satisfies $a^2>0$, we have $a\in \mathcal{C}_{K}\cup\mathcal{C}_{-K}$ for some\footnote{Actually, \cite{LiLiu01} also showed that $K$ is unique up to sign for minimal manifolds.} $K\in \mathcal{K}_X$. When $X$ is not minimal with $\mathcal{E}_X=\{\pm E_1,\cdots,\pm E_l\}$ and not diffeomorphic to rational or ruled manifolds, by Theorem \ref{thm:minimalmodel}, we can obtain its unique minimal model $\hat{X}$ by blowing down the exceptional spheres of the classes in $\mathcal{E}_X$. Choose any $K_{\hat{X}}\in\mathcal{K}_{\hat{X}}$ so that both $K_+:=K_{\hat{X}}+\sum \text{PD}(E_i)$ and $K_-:=-K_{\hat{X}}+\sum \text{PD}(E_i)$ are in $\mathcal{K}_X$. Assume $a\in H^2(X;\RR)$ satisfies $a^2>0$ and $a\cdot E_i>0$ for all $1\leq i\leq l$. Then it follows from the light cone lemma that $a\in \mathcal{C}_{K_+}\cup\mathcal{C}_{K_-}$.

\subsection{Biran's result of symplectic ball packings in dimension $4$}
Now, we review the result from \cite{biranpacking}, based on the pioneer work \cite{MP94}. Recall that a symplectic $4$-manifold $(X,\omega)$ is said to have {\bf SW-simple type} if its only nonzero Gromov invariants occur in classes $A\in H_2(X;\ZZ)$ for which $A^2-K_{\omega}\cdot A=0$. It is known (see \cite{Mcduffisotopy}) that the operation of blowup will preserve the property of being not of SW-simple type. Moreover, in the following situations, $X$ is not of SW-simple type:
\begin{itemize}
    \item $X$ is diffeomorphic to rational or ruled manifolds;
    \item $b_1(X)=0$ and $b_2^+(X)=1$;
    \item $b_1(X)=2$ and $(H^1(X))^2\neq 0$.
\end{itemize}

Let $(\tilde{X},\tilde{\omega})$ be some one-point symplectic blowup of $(X,\omega)$ and \[pr:H_2(\tilde{X};\ZZ)\rightarrow H_2(X;\ZZ)\] be the natural projection map. Consider the set
\[\mathcal{E}'_{\omega}:=pr(\mathcal{E_{\tilde{\omega}}})\setminus\{0\}\subseteq H_2(X;\ZZ)\]
and the value
\[d'_{\omega}:=\inf_{B\in \mathcal{E}'_{\omega}}\frac{\omega(B)}{c_1(B)-1}\in[0,\infty],\]
where the convention $\inf\emptyset=\infty$ is used. Biran's packing result in the special case of one ball can then be stated as the following.

\begin{theorem}[\cite{biranpacking}]\footnote{This is Theorem 6.A in the published version of \cite{biranpacking}. The arXiv version does not have section 6.}\label{thm:biranpacking}
    If $(X,\omega)$ is a symplectic $4$-manifold which is not of SW-simple type. Then $w_G(X,\omega)=\min\{\sqrt{[\omega]^2},d_{\omega}'\}$.
\end{theorem}

When $X$ is further assumed to be not diffeomorphic to rational or ruled manifolds, (\ref{equation:-1}) tells us the set $\mathcal{E}'_{\omega}$ must be either empty (when $X$ is minimal) or only contain exceptional classes $E$ in $\mathcal{E}_{\omega}$ which all have $c_1(E)=1$. Therefore, $d'_{\omega}=\infty$ in this case and we actually have the following corollary.

\begin{corollary}\label{cor:biranpacking}
    Let $(X,\omega)$ be a symplectic $4$-manifold which is not of SW-simple type and not diffeomorphic to rational or ruled manifolds. Then $w_G(X,\omega)=\sqrt{[\omega]^2}$.
\end{corollary}

\begin{rmk}
    \cite[Theorem 2.F]{biranpacking} already points out that if $X$ is minimal, not of SW-simple type and not diffeomorphic to rational or ruled manifolds, then $(X,\omega)$ admits a full packing by one ball. As shown above, the minimality assumption can actually be removed  in the non-rational or ruled cases by the observation (\ref{equation:-1}). This observation is also used in \cite{LiLiu01} to establish Theorem \ref{thm:liliucone}.
\end{rmk}

\begin{rmk}
    Note that cohomologous symplectic forms on rational or ruled manifolds are symplectomorphic (\cite[Theorem 3.18]{Li08}) and thus have the same Gromov width. Therefore, Corollary \ref{cor:biranpacking} also implies that the candidates of closed $4$-manifolds admitting cohomologous symplectic forms with different Gromov widths must have SW-simple type (such as $K3,T^4$).
\end{rmk}




\subsection{Uniruled symplectic manifolds}

\begin{definition}
For semipositive symplectic manifold $(X,\omega)$, a non-zero class $A\in H_2(X;\ZZ)$ is called a {\bf uniruled class} if there is a non-trivial genus zero Gromov-Witten invariant with a point insertion 
$$\text{GW}_{A,k}^X(\text{PD}(pt),a_2,\cdots,a_k)\neq 0.$$
$(M,\omega)$ is said to be {\bf symplectically uniruled} if there exists a uniruled class.
\end{definition}

The following celebrated result goes back to \cite{Gromov85}.

\begin{theorem}[Gromov]\label{thm:gromov}
    For any uniruled class $A$, $w_G(X,\omega)\leq \omega(A)$.
\end{theorem}

In dimension $4$, by the fundamental work \cite{McDuffrationalruled}, it is well known (\cite[Theorem 7.3]{Wendlbook}) that the characterization of symplectic uniruledness is equivalent to the underlying smooth $4$-manifold $X$ being rational or ruled manifolds. Namely, $X$ must be diffeomorphic to the blowup of $\CC\PP^2$ or $S^2$-bundle over Riemann surface. Moreover, when $X\neq \CC\PP^2$, the uniruled class $A$ can be chosen to be the `fiber class' in the sense that $A$ can be represented by a symplectically embedded sphere of self-intersection zero. The following statement will be used for the upper bound of Gromov width.

\begin{proposition}\label{prop:upperbound}
    Let $(X,\omega)$ be a symplectic rational manifold where $X=\CC\PP^2\#k\overline{\CC\PP^2}$. Then there always exists a uniruled class $A\in H_2(X;\ZZ)$ satisfying $A^2=0$ such that
    \begin{itemize}
        \item when $1\leq k\leq 4$, $\omega(A)<\sqrt{[\omega]^2}$; 
        \item when $k\geq 5$, there exists another symplectic form $\omega'$ deformation equivalent to $\omega$ satisfying $\omega'(A)<\sqrt{[\omega']^2}$ and $[\omega']\cdot K_{\omega'}<0$\footnote{This automatically holds when $k\leq 9$ by light cone lemma.}.
    \end{itemize}
\end{proposition}

\begin{proof}
Choose the standard basis $\{H,E_1,\cdots,E_k\}\subseteq H_2(X;\ZZ)$, where $H^2=1$ and all $E_i^2=-1$. Let $a=\omega(H)$ and $b_i=\omega(E_i)$. By the main result in \cite{KK17}, we may assume the period of $\omega$ satisfies the reduced condition\footnote{When $k=1,2$, the first equality is $a\geq b_1$ or $a\geq b_1+b_2$.}:
\[a\geq b_1+b_2+b_3,\,\,\,b_1\geq b_2\geq \cdots \geq b_{k}>0.\]
In this case, the canonical class $K_{\omega}=\text{PD}(-3H+\sum E_i)$. Then we can take the fiber class $H-E_1$ to be the uniruled class $A$ and compute
\[[\omega]^2-(\omega(A))^2=a^2-\sum_{i=1}^k b_i^2-(a-b_1)^2=2b_1(1-b_1)-\sum_{i=2}^kb_i^2\geq 2b_1(b_2+b_3)-\sum_{i=2}^kb_i^2.\]
It is clear that if $k\leq 4$, $2b_1(b_2+b_3)-\sum_{i=2}^kb_i^2$ will be positive by the reduced condition. If $k\geq 5$, in order to guarantee the positivity, we just need to deform the values of $b_5,\cdots,b_k$ so that they are as small as possible. This can always be realized by a symplectic deformation equivalence by \cite[Example 3.9+3.10]{Salamonsurvey}. Moreover, when $b_5,\cdots, b_k$ are very small, it also follows from the reduced condition that
\[[\omega']\cdot K_{\omega'}=-3a+b_1+b_2+b_3+b_4+\sum_{i\geq 5} b_i<0.\]
\end{proof}

Finally, we need the fact that the uniruled class can be lifted after taking the product with some symplectic $(S^2,\omega_{S^2})$. Indeed, \cite{LRuniruled} investigates when the existence of uniruled symplectic submanifolds with codimension $2$ implies the uniruledness of the ambient manifold. The following proposition is a very simple case. For example, see \cite[Proposition 2.10]{LRuniruled}.

\begin{proposition}\label{prop:uniruledindim6}
    Let $A\in H_2(X;\ZZ)$ be a uniruled class of $(X,\omega)$. Then, under the natural inclusion $H_2(X;\ZZ)\subseteq H_2(X\times S^2;\ZZ)$, $A$ is also a uniruled class for $(X\times S^2,\omega\oplus \omega_{S^2})$.
\end{proposition}

\begin{rmk}
  The obstruction coming from `stabilized' $J$-holomoprhic curves for symplectic embeddings was also used in the recent work \cite{siegelyao}. For any Riemann surface $\Sigma$, the problem of embedding $\sqcup B^4(R_i)\times \Sigma\hookrightarrow B^4(R)\times \Sigma$ was considered there.
\end{rmk}

\section{Dimension $6$}

From now on, $(X,\omega)$ and $(X',\omega')$ will be two simply-connected closed symplectic $4$-manifolds which are homeomorphic to each other. By Wall's result \cite{Wall64}, $X$ and $X'$ are $h$-cobordant and so are $X\times S^2$ and $X'\times S^2$. Then Smale's $h$-cobordism theorem \cite{smale} implies that there will be an orientation-preserving diffeomorphism between $X\times S^2$ and $X'\times S^2$. For our purpose of choosing cohomologous symplectic forms, we still need to keep track of the cohomological action of the diffeomorphism. The following classification result tells us that the category of torsion-free simply-connected smooth $6$-manifolds are equivalent to the category of cohomology rings with certain compatible characteristic classes. See also \cite{Ruan3fold}.

\begin{theorem}[\cite{Jupp}]\label{thm:jupp}
    Two closed simply-connected oriented smooth $6$-manifolds $M,N$ with torsion-free homology are oriented diffeomorphic if and only if there is an algebraic isomorphism $\Phi:H^*(M;\ZZ)\rightarrow H^*(N;\ZZ)$ which preserves the cup product $\mu:H^2\otimes H^2\otimes H^2\rightarrow\ZZ$, second Stiefel-Whitney class $w_2$ and first Pontryagin class $p_1$. Furthermore, the diffeomorphism can be chosen to realize this algebraic isomorphism $\Phi$.
\end{theorem}

We will apply the above theorem for $M=X\times S^2,N=X'\times S^2$. Note that (similarly for $X'\times S^2$)
\[H^2(X\times S^2;\ZZ)=H^2(X;\ZZ)\oplus \ZZ[\text{PD}(X)],\]
\[H^4(X\times S^2;\ZZ)=\ZZ[\text{PD}(S^2)]\oplus (H^2(X;\ZZ)\otimes \ZZ[\text{PD}(X)]),\]
\[w_2(X\times S^2)=\ZZ_2\text{-reduction of }c_1(\omega),\]
\[p_1(X\times S^2)=3\sigma(X)\text{PD}(S^2).\]
If $$\phi:H^2(X';\ZZ)\rightarrow H^2(X;\ZZ)$$ is any isomorphism preserving the intersection pairings, then we can naturally extend $\phi$ to an isomorphism $$\Phi:H^*(X'\times S^2;\ZZ)\rightarrow H^*(X\times S^2;\ZZ)$$ which maps $$\alpha,\text{PD}(X'),\text{PD}(S^2),\alpha\otimes \text{PD}(X')$$ to $$\phi(\alpha),\text{PD}(X),\text{PD}(S^2),\phi(\alpha)\otimes \text{PD}(X)$$ respectively, for any $\alpha\in H^2(X';\ZZ)$. Since $\phi$ maps characteristic elements to characteristic elements in the intersection form, the $\ZZ_2$-reduction of $\Phi$ will preserve $w_2$. $\Phi$ also preserves $p_1$ since it maps $\text{PD}(S^2)$ to $\text{PD}(S^2)$. Consequently, the following corollary is immediate.

\begin{corollary}\label{cor:jupp}
 Any isomorphism \[\phi\oplus \text{Id}:H^2(X';\ZZ)\oplus \ZZ[\text{PD}(X')]\rightarrow H^2(X;\ZZ)\oplus \ZZ[\text{PD}(X)],\] where $\phi$ preserves intersection pairings and $\text{Id}$ maps $\text{PD}(X')$ to $\text{PD}(X)$, can be realized as the induced map of an orientation-preserving diffeomorphism between $X\times S^2$ and $X'\times S^2$.
\end{corollary}

\begin{rmk}
    Corollary \ref{cor:jupp} is about the existence of such a diffeomorphism inducing the desired cohomology map. Conversely, it has been observed in \cite{HW23} that under the condition $\sigma(X)\neq 0$, the induced cohomology map of any diffeomorphism must also split as $\phi\oplus \pm\text{Id}$, where $\phi$ may reverse the sign of intersection pairings. This is also alluded in \cite[Example 9.7.1]{MSJcurve}. 
\end{rmk}

\section{Proof of Theorem \ref{thm:main}}
\begin{proof}[Proof of Theorem \ref{thm:main}]
    Let $(X,\omega)$ be a symplectic $4$-manifold homeomorphic but not diffeomorphic to the rational manifold $\CC\PP^2\#k\overline{\CC\PP^2}$. By Theorem \ref{thm:minimalmodel}, we can take its unique minimal model $(\hat{X},\hat{\omega})$ by blowing down $l$ disjoint symplectic exceptional spheres in classes $E_1,\cdots,E_l\in H_2(X;\ZZ)$. Let $\hat{X'}$ be the rational manifold homeomorphic to $\hat{X}$ and let $X':=\hat{X}\#l\overline{\CC\PP^2}$. Denote the line classes of $l\overline{\CC\PP^2}$ in $X'$ by $E'_1\cdots,E'_l\in H_2(X';\ZZ)$. By Proposition \ref{prop:upperbound}, we can choose some symplectic form $\omega'$ on $X'$ such that there is a uniruled class $A$ with $\omega'(A)<\sqrt{[\omega']^2}$. We further may assume $K_{\omega'}=K_{\hat{X'}}+\sum \text{PD}(E'_i)$ for some $K_{\hat{X'}}\in \mathcal{K}_{\hat{X'}}$ by first choosing symplectic form on $\hat{X'}$ and then performing blowup with small size to obtain $\omega'$. By Corollary \ref{cor:jupp}, there are orientation-preserving diffeomorphisms 
    \[f:\hat{X}\times S^2\rightarrow \hat{X'}\times S^2,\,\,\, \tilde{f}: X\times S^2\rightarrow X'\times S^2\]
    such that $f^*=\phi\oplus \text{Id},\tilde{f}^*=\tilde{\phi}\oplus \text{Id}$, where $\phi:H^2(\hat{X'};\ZZ)\rightarrow H^2(\hat{X};\ZZ)$ is an isomorphism preserving the intersection pairings and $\tilde{\phi}$ is the extension of $\phi$ which maps each $\text{PD}(E_i')$ to $\text{PD}(E_i)$. Let us choose two classes in $\mathcal{K}_X$ \[K_+:=K_{\hat{\omega}}+\sum \text{PD}(E_i),\,\,\,K_{-}:=-K_{\hat{\omega}}+\sum \text{PD}(E_i)\] and consider the class\footnote{Here, we also use $\tilde{\phi}$ to denote the map between real cohomology groups.} $\tilde{\phi}([\omega'])$.  Since $[\omega']\in\mathcal{C}_{K_{\omega'}}$ and $\mathcal{E}_{K_{\omega'}}\supseteq\{E_1',\cdots,E_l'\}$, by Theorem \ref{thm:liliucone}, $\omega'(E_i')>0$ for all $i$. Also, note that $\mathcal{E}_{K_+}=\{E_1,\cdots,E_l\}$ by (\ref{equation:-1}) and for all $i$, \[\tilde{\phi}([\omega'])(E_i)=\tilde{\phi}([\omega'])\cdot \tilde{\phi}(\text{PD}(E_i'))=[\omega']\cdot\text{PD}(E_i')=\omega'(E_i')>0.\] 
    Hence, the discussion after Theorem \ref{thm:liliucone} implies $\tilde{\phi}([\omega'])\in\mathcal{C}_{K_+}\cup\mathcal{C}_{K_-}$ so that there will be some symplectic form $\omega_X$ (different from the initial $\omega$) on $X$ in this class.

    Now, let us choose a symplectic form $\omega_{S^2}$ on $S^2$ with sufficiently large area and compare the Gromov widths of the cohomologous symplectic forms $\omega_X\oplus \omega_{S^2}$ with $\tilde{f}^*(\omega'\oplus \omega_{S^2})$. By Corollary \ref{cor:biranpacking}, \[w_G(X,\omega_X)=\sqrt{[\omega_X]^2}=\sqrt{[\tilde{\phi}([\omega'])]^2}=\sqrt{[\omega']^2}.\]
    Since $\omega_{S^2}$ has large enough area, we see that \[w_G (X\times S^2,\omega_X\oplus \omega_{S^2})\geq \sqrt{[\omega']^2}.\] On the other hand, by Theorem \ref{thm:gromov} and Proposition \ref{prop:uniruledindim6}, we also have \[w_G (X\times S^2,\tilde{f}^*(\omega'\oplus \omega_{S^2}))=w_G(X'\times S^2,\omega'\oplus \omega_{S^2})\leq (\omega'\oplus \omega_{S^2})(A)=\omega'(A)<\sqrt{[\omega']^2}.\]
    Therefore, they have different Gromov widths. 
    
    Fianlly, if these two cohomologous symplectic forms $\omega_X\oplus \omega_{S^2}$ with $\tilde{f}^*(\omega'\oplus \omega_{S^2})$ have the same first Chern classes, then $\tilde{\phi}(K_{\omega'})=K_{\omega_X}$. Recall that by Proposition \ref{prop:upperbound}, we may also assume $[\omega']\cdot K_{\omega'}<0$. Thus,\[[\omega_X]\cdot K_{\omega_X}=\tilde{\phi}([\omega'])\cdot \tilde{\phi}(K_{\omega'})=[\omega']\cdot K_{\omega'}<0. \] This will contradict with Liu-Ohta-Ono's Theorem \cite{Liu,OhtaOno} which states that the existence of such symplectic form $\omega_X$ implies the manifold $X$ is diffeomorphic to rational or ruled manifold. So their first Chern classes are also different.
\end{proof}

Note that it has been observed in \cite[Proposition 4.10]{Li08} that $\mathcal{C}_{X}=\mathcal{P}$ when $X=K3$. Indeed, every positive class even has a hyperK\"ahler representative. Therefore, in the above proof, if we take $(X',\omega')$ to be any homotopy $K3$ with $K_{\omega'}\neq 0$ and consider the diffeomorphism $f:X\times S^2\rightarrow X'\times S^2$, there will be some symplectic form $\omega_X\oplus \omega_{S^2}$ on $X\times S^2$ cohomologous to $f^*(\omega'\oplus \omega_{S^2})$. However, the Chern class for any symplectic form on K3 must vanishi. As a result, $c_1(\omega_X\oplus \omega_{S^2})$ belongs to the subspace $\ZZ[\text{PD}(X)]\subseteq H^2(X\times S^2;\ZZ)$ while $c_1(f^*(\omega'\oplus \omega_{S^2}))$ not. They must be different.
 
Finally, we remark that one can easily obtain examples of cohomologous symplectic forms with different Gromov widths and first Chern classes in higher dimensions by taking more products. For instance, consider $(X\times S^2)\times (S^2)^k$ with cohomologous symplectic forms $\omega_X\oplus \omega_{S^2}\oplus(\omega_{S^2})^k$ and $\tilde{f}^*(\omega'\oplus \omega_{S^2})\oplus (\omega_{S^2})^k$.

\bibliographystyle{amsalpha}
\bibliography{mybib}{}

\end{document}